\documentclass[11pt]{amsart}
\usepackage{amsmath,amsfonts,amssymb}
\usepackage[mathscr]{eucal}
\usepackage{xcolor}
\usepackage{enumitem}
\usepackage{hyperref}
\hypersetup{colorlinks=true,linktoc=all,linkcolor=blue,citecolor=blue,urlcolor=blue,
  pdftitle={Fundamental groups of Calabi-Yau spaces},
  pdfauthor={Xin Fu, Bin Guo, Jian Song and Juanyong Wang}}
\allowdisplaybreaks[1]
\setlist[enumerate]{leftmargin=*,label=\textup{(\roman*)},itemsep=3pt}
\newtheorem{theorem}{Theorem}[section]
\newtheorem{proposition}[theorem]{Proposition}
\newtheorem{lemma}[theorem]{Lemma}

\theoremstyle{remark}
\newtheorem{remark}[theorem]{Remark}
\numberwithin{equation}{section}
\newcommand{\R}{\mathbb R}
\newcommand{\C}{\mathbb C}
\newcommand{\Q}{\mathcal Q}
\newcommand{\reg}{\mathrm{reg}}
\newcommand{\loc}{\mathrm{loc}}
\newcommand{\eps}{\varepsilon}
\newcommand{\dd}{\mathrm d}
\newcommand{\ddc}{\sqrt{-1}\partial\bar\partial}
\newcommand{\Ric}{\operatorname{Ric}}
\newcommand{\tr}{\operatorname{tr}}
\newcommand{\Hess}{\operatorname{Hess}}
\newcommand{\RCD}{\operatorname{RCD}}
\newcommand{\BE}{\operatorname{BE}}

\newcommand{\Z}{\mathbb Z}
\newcommand{\cO}{\mathcal O}
\newcommand{\tX}{\widetilde X}
\newcommand{\Lip}{\operatorname{Lip}}
\newcommand{\LL}{\mathcal L}

\title[Fundamental groups of Calabi-Yau spaces]
{Fundamental groups of Calabi-Yau spaces}
\author[X. Fu]{Xin Fu}
\address{School of Science, Institute for Theoretical Sciences,
Westlake University, Hangzhou 310030, China}
\email{fuxin54@westlake.edu.cn}
\thanks{Xin Fu is supported by National Key R\&D Program of China
2024YFA1014800 and NSFC No.~12401073.}
\author[B. Guo]{Bin Guo}
\address{Department of Mathematics \& Computer Science,
Rutgers University, Newark, NJ 07102, USA}
\email{bguo@rutgers.edu}
\thanks{Bin Guo is  supported in part by the National
Science Foundation grant DMS-2203607.}
\author[J. Song]{Jian Song}
\address{Department of Mathematics, Rutgers University,
Piscataway, NJ 08854, USA}
\email{jiansong@math.rutgers.edu}
\thanks{ Jian Song is supported in part by the National Science Foundation grant DMS-2505575.}
\author[J. Wang]{Juanyong Wang}
\address{State Key Laboratory of Mathematical Sciences,
Academy of Mathematics and Systems Science, Chinese Academy of Sciences,
Beijing 100190, China}
\email{juanyong.wang@amss.ac.cn}
\thanks{Juanyong Wang is supported by NSFC (Grants No.~12301060 and
12288201) and National Key R\&D Program of China
(Grant No.~2021YFA1003100).}
\date{}

\begin{document}
\raggedbottom
\clubpenalty=10000
\widowpenalty=10000
\displaywidowpenalty=10000
\begin{abstract}
We study the fundamental groups of a normal projective Calabi-Yau
variety $X$ with klt singularities and of its regular locus $X_{\reg}$.
We prove that both groups are almost abelian, and that their ranks
are determined by the irregularities of finite \'etale and
quasi-\'etale covers, respectively. Our approach is differential
geometric, and relies on recent advances in the geometric theory of
complex Monge-Amp\`ere equations.

\end{abstract}
\maketitle

\section{Introduction}\label{sec:introduction}

It is a classical theme in K\"ahler geometry that the fundamental
group of a compact K\"ahler manifold with vanishing first Chern class
is governed by its Ricci-flat geometry. By Yau's solution of the
Calabi conjecture \cite{Yau}, every K\"ahler class contains a
Ricci-flat K\"ahler metric. The splitting theorem of Cheeger-Gromoll
\cite{CG}, combined with the K\"ahler structure, then shows that the
universal cover is of the form $\C^q\times Y$, where $Y$ is compact
and simply connected. Consequently, after passing to a finite
\'etale cover, the manifold becomes a product of a complex torus and
such a factor $Y$, and its fundamental group is almost
abelian. The decomposition theory of Bogomolov and Beauville
\cite{Bogomolov,Beauville} goes further, and describes $Y$ as a
product of irreducible Calabi-Yau and irreducible holomorphic
symplectic manifolds. This is a basic instance of a general
principle, namely that Ricci curvature can determine both the
topology and the complex structure of a compact space.

The question of fundamental groups belongs more broadly to the study
of compact K\"ahler manifolds with nef anticanonical bundle. 
P\u aun \cite{Paun97} proved that the fundamental groups of such
manifolds have polynomial growth, and that they are almost abelian
in the projective case. Combined with his later surjectivity theorem
for the Albanese map \cite{Paun17}, his work yields almost
abelianity in the compact K\"ahler case as well. These results show
that the conclusion on the fundamental group persists well beyond
the existence of a Ricci-flat metric. Extensions to singular spaces
and log pairs have recently been studied by Fu-Guo-Song-Wang
\cite{FGSW}. In the present paper, we focus on projective Calabi-Yau
spaces, and on the ordinary fundamental group of their regular loci.

Related structure theorems have been obtained by algebraic methods,
using positivity of direct images and the theory of foliations.
Cao \cite{CaoAlbanese} proved local triviality of the Albanese map
for projective manifolds with nef anticanonical bundle, and
Cao-H\"oring \cite{CaoHoering} established their decomposition
theorem. Campana-Cao-Matsumura \cite{CCM} extended the structure
theorem to klt pairs with smooth underlying variety.
Matsumura-Wang \cite{MatsumuraWang} subsequently proved that a
projective klt pair with nef anti-log canonical divisor admits,
after a finite quasi-\'etale cover, a locally trivial rationally
connected fibration over a variety with numerically trivial canonical
class. More recently, Matsumura-Wang-Wu-Zhang \cite{MWWZ}
established a Beauville-Bogomolov-Yau decomposition for K\"ahler
generalized klt pairs of Calabi-Yau type, with a locally constant
rationally connected fibration after a finite quasi-\'etale cover.
These results place the fundamental group problem within the broader
structure theory of varieties and pairs with nef anticanonical class.

For singular Calabi-Yau varieties, an algebraic decomposition theory
has been developed through the work of Druel \cite{Druel},
Greb-Guenancia-Kebekus \cite{GGK}, and H\"oring-Peternell
\cite{HP}. In the projective case, it shows that, after a finite
quasi-\'etale cover, the variety becomes a product of an abelian
variety and singular analogues of the irreducible factors in the
smooth decomposition. In the compact K\"ahler setting,
Claudon-Graf-Guenancia-Naumann \cite{CGGN} developed the Bochner
principle and the Albanese theory, and Bakker-Guenancia-Lehn
\cite{BGL} established the corresponding decomposition theorem. The
dimension of the abelian factor is equal to the augmented
irregularity. Earlier low-dimensional results for $\pi_1(X)$ include
\cite{GKPTrivial} for projective canonical varieties
with numerically trivial canonical class and dimension at most four,
\cite{Graf} for compact
K\"ahler threefolds of Kodaira dimension zero, and
\cite{CGGN} for compact K\"ahler spaces with klt
singularities, vanishing first Chern class and dimension at most four.
The study of the fundamental group of the regular part of Calabi-Yau varieties is substantially harder. 
Greb-Kebekus-Peternell \cite{GKP} established
extension results for finite \'etale covers of regular loci, while
\cite{GGK} imposes strong restrictions on the
linear representations of the fundamental group of the regular
locus when the augmented irregularity vanishes. The object of the
present paper is the ordinary topological fundamental group itself. In complex
dimension three, Braun-Figueroa \cite{BF} prove that
$\pi_1(X_{\reg})$ is almost abelian. To our knowledge, almost
abelianity of $\pi_1(X_{\reg})$ remains largely open in higher
dimensions.

Singularities present a genuine difficulty for the differential
geometric argument. The singular Ricci-flat currents constructed by
Eyssidieux-Guedj-Zeriahi \cite{EGZ} are smooth on the regular locus,
but this locus is in general incomplete. Without control of the
metric boundary, neither the classical splitting theorem nor a
global integration by parts argument can be applied there.
Moreover, the universal cover of the regular locus need not extend
as an unramified cover of the whole variety. It is therefore
necessary to understand both the metric completion and the local
monodromy around the singular set.

The geometric theory of complex Monge-Amp\`ere equations provides
the analytic tools for this approach. Early developments include
Song's study of the Riemannian geometry of K\"ahler-Einstein currents
\cite{SongKE}, which identified the metric completions with the
underlying projective varieties in the Calabi-Yau and canonically
polarized settings with crepant singularities. Fu-Guo-Song
\cite{FGSGeometric} established uniform gradient and diameter
estimates for geometric complex Monge-Amp\`ere equations, with
applications to singular solutions and collapsing families of twisted
K\"ahler-Einstein metrics. These works initiated a systematic study
of the metric geometry determined by degenerate Monge-Amp\`ere
equations. Guo-Phong-Song-Sturm
\cite{GPSSDiameter,GPSSDiameterII,GPSS} subsequently established
diameter and Sobolev estimates under entropy bounds, providing
analytic tools on singular K\"ahler spaces.

Sz\'ekelyhidi \cite{Szekelyhidi} proved that singular
K\"ahler-Einstein metrics define RCD spaces under a constant scalar
curvature approximation hypothesis. Fu-Guo-Song \cite{FGS}
obtained the RCD structure in complex dimension three. Related RCD
criteria under regularization hypotheses appear in
\cite{GS} and
\cite{CCHSTT}.
Sz\'ekelyhidi \cite{SzekelyhidiCollapsing} further
identified the metric completion with the underlying complex space
under the stated analytic and RCD hypotheses, with applications to
collapsing Calabi-Yau fibrations. More recently, Guo-Song-Sturm \cite{GSS} obtained the required RCD
structure in all dimensions by means of Ricci entropy estimates. This last result allows the comparison and splitting methods of
Riemannian geometry to be applied to singular Calabi-Yau spaces.

The relevant non-smooth tools include Gigli's differential calculus
and splitting theorem \cite{GigliCalculus,Gigli}, the splitting
principle of Gigli-Marconi \cite{GM}, and the covering space theory
of Mondino-Wei \cite{MW}. Honda-Sun \cite{HS} give criteria for
extending a lower bound on the Ricci curvature from an almost smooth
space to its completion. Together, these results provide the
analytic framework for the proofs below. A particular feature of our
setting is that the K\"ahler structure allows us to upgrade a real
RCD splitting to a holomorphic metric product.

\bigskip

Throughout this paper, a \emph{projective Calabi-Yau variety} is a
connected normal projective variety $X$ over $\C$ with klt
singularities and $K_X\equiv0$. No condition is imposed on its
irregularity. By abundance in numerical dimension zero, we have
$K_X\sim_{\mathbb Q}0$ \cite{Gongyo}. We fix an ample line bundle $L$
and let $\omega\in c_1(L)$ be its singular Ricci-flat K\"ahler metric
\cite{EGZ}. This current has bounded local potentials, and it is
smooth and Ricci-flat on the analytic regular locus $X_{\reg}$.
Wedge products of currents with bounded local potentials are
understood in the sense of Bedford-Taylor \cite{EGZ}.
By \cite{GSS}, the completion of its intrinsic
distance is canonically homeomorphic to $X$, and, equipped with the
measure $\mu=\omega^n/n!$, it is an $\RCD(0,2n)$ space, where
$n=\dim_\C X\ge1$. All fundamental groups in this paper are ordinary
topological fundamental groups. A zero-dimensional factor $Y$ below
is understood to be a point.

Let $\pi:\tX\to X$ be the ordinary universal cover, endowed with the
lifted complex structure, length metric and measure. Our first result
is a complex splitting theorem for the K\"ahler RCD spaces arising
in this setting, which is inspired by Cheeger-Colding-Tian \cite{CCT} in the Ricci limit space case.

\begin{theorem}\label{thm:complex}
Let $X$ be an $n$-dimensional normal projective Calabi-Yau variety, equipped
with the singular Ricci-flat K\"ahler metric of a fixed polarization.
If the universal cover $\tX$ contains
a metric line, then there is a holomorphic and isometric  splitting
\begin{equation}\label{eq:complex-splitting}
 \tX\simeq\C\times Y,
\end{equation}
where $Y$ is a simply connected normal K\"ahler space with klt
singularities of dimension $n-1$.  
\end{theorem}

For a normal projective variety $V$, we write $q(V)=h^1(V,\cO_V)$, and set
\begin{align}
 q'(X)&=\sup\{q(X'):X'\to X\text{ connected finite \'etale}\},
                                  \label{eq:etale-irregularity}\\
 \widetilde q(X)&=\sup\{q(X'):X'\to X\text{ connected finite quasi-\'etale}\}.
                                  \label{eq:quasi-irregularity}
\end{align}
Recall that a finite surjective map of normal varieties is said to be
quasi-\'etale if it is \'etale away from a subset of complex
codimension at least two in the target. The second invariant is the
augmented irregularity.

Our second result extends the classical structure theorem for smooth
Calabi-Yau manifolds directly to the whole singular space.

\begin{theorem} 
\label{thm:etale}
Let $X$ be a normal projective Calabi-Yau variety with klt singularities. Then there is a finite
\'etale cover
\[
 A\times Y\longrightarrow X,
\]
where $A$ is an abelian variety of dimension $q'(X)$ and
$Y$ is a projective Calabi-Yau variety with $\pi_1(Y)=\{1\}$. 
Moreover,
\[
 \tX\simeq\C^{q'(X)}\times Y
\]
holomorphically, and $\pi_1(X)$ contains a subgroup
of finite index isomorphic to $\Z^{2q'(X)}$.
\end{theorem}

In fact, Guo-Song-Sturm have already proved that $\pi_1(X)$ is
almost abelian \cite{GSS} based on the results of Fu-Guo-Song-Wang \cite{FGSW}. The proof of
Theorem~\ref{thm:etale} provides in addition the holomorphic and isometric product structure,
and identifies the rank through finite \'etale covers. Once
Theorem~\ref{thm:complex} and the structure theorem for compact RCD
spaces are available, its proof follows the argument in the smooth
case.

Our main result concerns the fundamental group of the regular locus,
whose study requires a further analysis of the completed universal
cover.

\begin{theorem} 
\label{thm:decomposition}
Let $X$ be a normal projective Calabi-Yau variety with klt singularities. Then there is a finite
quasi-\'etale cover
\[
 A\times Y\longrightarrow X,
\]
where $A$ is an abelian variety and $Y$ is a normal projective variety
with canonical singularities, satisfying
\[
 \dim A=\widetilde q(X),\qquad K_Y\sim0,\qquad
 \widetilde q(Y)=0,\qquad \pi_1(Y_{\reg})=\{1\}.
\]
Furthermore, $\pi_1(X_{\reg})$ contains a subgroup of finite index
isomorphic to $\Z^{2\widetilde q(X)}$.
\end{theorem}

 We stress that the two covers need not have the same abelian factor:
the first is \'etale everywhere, including above the singular points,
while the second may ramify over the singular locus.   In particular, $\widetilde q(X)=0$ implies
that $\pi_1(X_{\reg})$ is finite.

The existence of an abelian factor of dimension $\widetilde q(X)$
is part of the earlier decomposition theory \cite{CGGN,HP,BGL}.
The new topological conclusion in Theorem~\ref{thm:decomposition} is
the simple connectedness of the regular locus of the complementary
factor. It determines the ordinary group $\pi_1(X_{\reg})$ up to
finite index, and the construction is compatible with the prescribed
Ricci-flat metric. Throughout the proofs, the Bochner principle of
\cite{CGGN} is used to control holomorphic vector fields on the
compact factor.

The key new idea beyond Theorem~\ref{thm:decomposition} is to work with the
metric completion of the universal cover of the regular locus,
\[
 p:M\longrightarrow X_{\reg},\qquad
 \widehat M=\overline{(M,p^*\omega)}.
\]
The main algebraic geometric input in this step is Braun's local
finiteness theorem \cite{Braun}. The finiteness of the local
algebraic fundamental group of a klt singularity was first proved by
Xu \cite{Xu}, and Tian-Xu \cite{TX} developed an inductive approach
to topological finiteness. Braun then proved that
$\pi_1(U\cap X_{\reg})$ is finite for sufficiently small
neighborhoods $U$. It follows that each component of
$p^{-1}(U\cap X_{\reg})$ has finite degree over $U\cap X_{\reg}$,
and the Grauert-Remmert extension theorem then produces a normal
analytic extension $\widehat M_{\mathrm{an}}$, which is locally
finite over $X$. We shall prove that
\[
 \widehat M\simeq\widehat M_{\mathrm{an}},\qquad
 (\widehat M,d_M,\mu_M)\text{ is }\RCD(0,2n).
\]
The first assertion identifies the metric and analytic topologies.
The second is the main analytic difficulty, since the covering may
have infinitely many sheets, and the required global estimates must
be established on its completion. A related differential geometric
approach to Braun's local finiteness theorem, combining the techniques
developed here with the inductive scheme of Tian-Xu \cite{TX}, is
discussed in the private notes of Fu-Song \cite{FuSo}. The present
proof uses Braun's theorem as an external input.

We obtain these estimates by constructing smooth approximations on
a compact orbifold resolution, and by applying the Ricci entropy
calculations of Guo-Song-Sturm \cite{GSS}. Summing over the sheets
allows us to control scalar elliptic estimates on the cover by
Sobolev estimates on the compact quotient. The resulting gradient
bounds and capacity cutoffs then yield the global Bochner
inequality, following the argument of Honda-Sun \cite{HS}. Once
$\widehat M$ is known to be RCD, applying our complex RCD version of Cheeger-Gromoll type splitting 
Theorem~\ref{thm:etale} to the resulting compact quasi-\'etale cover
yields Theorem~\ref{thm:decomposition}. 

Extensions to klt pairs, and to normal projective varieties with nef
anticanonical divisor and klt singularities in the sense of
\cite{FGSW}, will be treated in a separate note, together with the
case of compact K\"ahler Calabi-Yau spaces. 
For the latter, the local
finiteness is supplied by Fu-Ou \cite{FuOu}, who
extend Braun's theorem to germs of complex analytic klt
singularities.

The paper is organized as follows. In Section~\ref{sec:complex}, we
prove the complex splitting theorem by constructing a global
holomorphic conjugate of a real splitting coordinate.
Theorem~\ref{thm:etale} is proved in Section~\ref{sec:etale}. In
Sections~\ref{sec:completion}-\ref{sec:lifting}, we identify the two
completions, construct the orbifold approximations, and establish the
RCD structure on $\widehat M$. Finally, the proof of
Theorem~\ref{thm:decomposition} is completed in
Section~\ref{sec:decomposition}.

\bigskip
\noindent\textbf{Acknowledgements.}
The authors would like to thank Chenyang Xu for bringing the problem
of fundamental groups of singular Calabi-Yau spaces to their
attention. The main results were first obtained in complex dimension
three in April 2026, building on the work of Fu-Guo-Song \cite{FGS},
and were announced at the Oberwolfach workshop in June 2026. They
were extended to higher dimensions after the work of
Guo-Song-Sturm \cite{GSS} appeared. AI tools were used only for
exposition and verification.

\bigskip
\bigskip

\section{A complex splitting theorem}\label{sec:complex}

\subsection{Affine functions and the global conjugate}

Since complex spaces are locally contractible, $\tX$ exists as an
ordinary simply connected cover. It is normal and klt, and
$\tX_{\reg}=\pi^{-1}(X_{\reg})$. The lifted length metric is complete,
induces the analytic topology, and is the completion of the intrinsic
Riemannian metric on $\tX_{\reg}$. These properties lift from $X$ by
working in evenly covered neighborhoods and subdividing curves. The
covering and local-to-global results for RCD spaces then give
\[
 (\tX,d,\mu)\text{ is }\RCD(0,2n),
\]
see \cite[Section 2.2 and Lemma 2.18]{MW}. On $\tX_{\reg}$, the RCD
differential calculus agrees with that of the smooth lifted
Ricci-flat metric.

We recall two local facts: the analytic singular set has zero local
$W^{1,2}$-capacity \cite[Lemma 8.2]{GPSS}, and
$\omega\ge c\theta$ for a smooth background K\"ahler form $\theta$
\cite[Theorem 1.1]{CCHSTT}; see also
\cite[Sections 2.2 and 7.5]{GSS}.
Both statements lift to evenly covered neighborhoods. In particular,
a local holomorphic function has bounded metric gradient on smaller
neighborhoods, and is therefore locally Lipschitz. Its real and
imaginary parts are weakly harmonic across the singular set: it
suffices to insert capacity cutoffs in the smooth harmonic equation,
and the bounded gradient forces the cutoff error to tend to zero.
This also identifies their classical and Sobolev differentials
almost everywhere.

\begin{remark} \label{sp:rem:affine-holomorphic}
Since the metric is smooth on $\tX_{\reg}$, a locally Lipschitz
RCD harmonic function $u$ with $\Hess u=0$ is smooth there by
elliptic regularity, and its classical Hessian vanishes. The
K\"ahler identities then imply that $u$ is pluriharmonic and that
$\nabla u$ is a parallel real holomorphic vector field.
\end{remark}

\begin{proposition}[Gigli, Gigli-Marconi]
\label{sp:prop:real-splitting}
Let $(Z,d,m)$ be a complete $\RCD(0,N)$ space, $1\le N<\infty$,
with full support. Suppose that a real function
$b\in H^{2,2}_{\loc}(Z)\cap D(\Delta_{\loc})$ satisfies
\begin{equation}\label{sp:eq:real-affine}
 |\dd b|=1,\qquad \Delta b=0,\qquad \Hess b=0.
\end{equation}
Then there exist a complete metric-measure space $(W,d_W,m_W)$ and
a measure-preserving isometry
\[
 (Z,d,m)\simeq(\R\times W,d_{\mathrm{prod}},\LL^1\otimes m_W)
\]
under which the $1$-Lipschitz representative of $b$ is the real
coordinate: $b(t,w)=t$.
\end{proposition}
\begin{proof}
This is \cite[Theorem 3.3]{GM} with both the Laplacian and Hessian
coefficients equal to zero, in which case the normalized warping
functions are identically one.
\end{proof}

By Gigli's splitting theorem \cite{Gigli}, a line in $\tX$ gives rise
to a real coordinate $b$ satisfying \eqref{sp:eq:real-affine}. The
same identities hold for the first coordinate of a prescribed real
splitting. We fix such a function $b$.

\begin{lemma} \label{sp:lem:conjugate}
There is a holomorphic function $f=b+ih$ on $\tX$ such that $h$ is
$1$-Lipschitz and
\begin{equation}\label{sp:eq:conjugate-global}
 |\dd h|=1,\qquad\Delta h=0,\qquad\Hess h=0,
 \qquad\langle\nabla b,\nabla h\rangle=0.
\end{equation}
On $\tX_{\reg}$, we have $\nabla h=J\nabla b$, and both gradient
fields extend as real holomorphic derivations on $\tX$.
\end{lemma}
\begin{remark}On a singular complex space, a derivation $D$ is a $\mathbb C$-linear map of sheaves satisfying
\[
D:\mathcal O_{\widetilde X}\longrightarrow\mathcal O_{\widetilde X},
\qquad
D(gu)=gD(u)+uD(g).
\]
\end{remark}
\begin{proof}
By Remark~\ref{sp:rem:affine-holomorphic}, $b$ is pluriharmonic on
$\tX_{\reg}$, and it is locally bounded on $\tX$. The extension of
bounded pluriharmonic functions on normal spaces
\cite[Corollary 2.2]{Fujiki} shows that it is locally the real part
of a holomorphic function, also near singular points. The
differences of these local functions are locally constant imaginary
numbers, and their Cech cohomology class in $H^1(\tX,i\R)$ vanishes by simple
connectedness and local contractibility. After adjusting constants,
they therefore glue to a global function $f=b+ih$.

On the smooth locus, the Cauchy-Riemann equations give
$\nabla h=J\nabla b$, so that $|\dd h|=1$ and the two gradients are
orthogonal. The argument for local holomorphic functions given above
extends harmonicity across the singular set, and the
Sobolev-to-Lipschitz property, together with the length-space
property, makes $h$ globally $1$-Lipschitz. Since $h$ is locally
bounded and Lipschitz and $\Delta h=0$, it is a local test function.
The local second-order regularity in \cite[Theorem 3.3.8]{GigliCalculus}
therefore gives $h\in H^{2,2}_{\loc}(\tX)$. The improved Bochner
inequality now gives, with the measure-valued Laplacian,
\[
 0=\tfrac12\Delta|\dd h|^2
       \ge |\Hess h|^2\,\mu,
\]
hence $\Hess h=0$. Thus Remark~\ref{sp:rem:affine-holomorphic}
applies to both functions. Finally, a holomorphic vector field on the
regular locus of a normal space extends as a derivation.
This proves the last assertion.
\end{proof}

\subsection{The holomorphic splitting}

\begin{proof}[Proof of Theorem~\ref{thm:complex}]
Let $f=b+ih$ be the function given by Lemma~\ref{sp:lem:conjugate}.
On the regular locus, $E=\nabla b$ and $H=\nabla h=JE$ are
orthonormal parallel real holomorphic vector fields. The holomorphic
vector field $V=\frac12(E-iH)$ extends as a derivation on $\tX$, and
$V(f)=1$ everywhere. By Kaup's local flow theorem
\cite[Satz 3]{Kaup} (see also \cite[Fact~4.2]{GKK}), $V$ admits a local holomorphic flow, whose real
and imaginary time directions are given by $E$ and $H$.

We claim that both real flows are complete. Indeed, their local maps
are biholomorphisms which preserve the regular locus, where the
trajectories have unit speed. Thus, for the $E$-flow,
\[
 d(\Phi_s x,\Phi_t x)\le |s-t|
\]
whenever the trajectory is defined between $s$ and $t$, and the same
bound holds for the $H$-flow. By continuity and density, these bounds
extend to trajectories through singular points. A finite maximal
endpoint would therefore have a limit in the complete space $\tX$.
Since the metric and analytic topologies agree, the local holomorphic
flow at that limit continues the trajectory, by uniqueness for
ordinary differential equations in a local analytic embedding.
This contradicts maximality, and proves the claim.

Denote the complete flows by $\Phi_s$ and $\Psi_t$.
Parallelness gives $[E,H]=0$ on the regular locus, hence everywhere
as derivations by normality. Their flows therefore commute and form
a holomorphic $\C$-action
$A_{s+it}=\Phi_s\Psi_t$. Since $E(f)=1$ and $H(f)=i$, we have
\[
 f(A_zx)=f(x)+z.
\]
In particular, $f$ is surjective to $\mathbb C$. We give $Y=f^{-1}(0)$ its analytic
fiber structure. The maps
\begin{equation}\label{sp:eq:product-inverse}
 F:(z,y)\longmapsto A_z(y),\qquad
 x\longmapsto\bigl(f(x),A_{-f(x)}x\bigr)
\end{equation}
are mutually inverse holomorphic maps. Thus $f$ is the projection of
a global holomorphic product $\tX\simeq\C\times Y$, with $b=\Re z$.
Normality and the klt property descend to $Y$.

On $\tX_{\reg}=\C\times Y_{\reg}$, the fields $E,H$ are orthogonal
to the fibers, and their flows preserve the metric. Hence
$F^*g=|\dd z|^2+g_Y$. Passing to intrinsic distances and metric
completions, we may identify $Y$ with the completion of
$(Y_{\reg},g_Y)$, and obtain
\begin{equation}\label{sp:eq:real-plane}
 (\tX,d,\mu)\simeq
 (\C\times Y,d_{\mathrm{prod}},\LL^2\otimes\mu_Y).
\end{equation}
The product identity also yields the displayed equality of volume
measures. The factor metric is smooth, K\"ahler and Ricci-flat on
$Y_{\reg}$. A lifted background K\"ahler form restricts to $Y$, and
since the product retracts onto $Y$, the latter is simply connected.
Applying the dimension drop in \cite{Gigli} successively to the two
real coordinate directions shows that $Y$ is $\RCD(0,2n-2)$ when
$n\ge2$. Of course when $n=1$, the factor is a point.
\end{proof}

\section{The fundamental group of the entire space}\label{sec:etale}

We will prove Theorem \ref{thm:etale} in this section. 

\subsection{Splitting structures}

The structure theorem for compact RCD spaces \cite[Theorem 1.3]{MW}
gives a splitting
\begin{equation}\label{sp:eq:mw-product}
 \tX=\R^r\times Y_0,
\end{equation}
where $Y_0$ is compact and RCD with nonnegative Ricci curvature.
Since $\tX$ is simply connected, $Y_0$ is
simply connected.

\begin{proposition}\label{sp:prop:maximal}
We have $r=2q$ for some $q\ge0$, and \eqref{sp:eq:mw-product} is a
holomorphic and isometric product
\begin{equation}\label{sp:eq:full-complex-product}
 \tX\simeq\C^q\times Y,
\end{equation}
where $Y$ is a compact simply connected normal K\"ahler space with
klt singularities and $c_1(Y)=0$.
\end{proposition}
\begin{proof}
Let $b$ be a unit linear coordinate of $\R^r$, and let $h$ be its
conjugate from Lemma~\ref{sp:lem:conjugate}. In the product
constructed in Theorem~\ref{thm:complex}, the complete holomorphic
flow of $h$ is a translation, and gives a line along which $h$
increases at unit speed. Since $Y_0$ is compact, every line in
$\R^r\times Y_0$ has the form $(a_0+tv,y_0)$. Computing the Busemann
function of this line in the two splittings, we find
\[
 h(a,y)-h(a_0,y_0)
 =\lim_{t\to\infty}\left(t-
       \sqrt{|a-a_0-tv|^2+d_{Y_0}(y,y_0)^2}\right)
 =\langle v,a-a_0\rangle.
\]
Thus $J$ preserves the space of Euclidean translation fields, and is
constant on that space. Being an orthogonal complex structure, it
forces $r=2q$. Pairing orthonormal real coordinates, we obtain
holomorphic coordinates $f_j=b_j+ih_j$, whose translations form a
holomorphic $\C^q$-action. The fiber of $(f_1,\ldots,f_q)$ at zero
then gives the product, via the same inverse maps as in
\eqref{sp:eq:product-inverse}. Metrically, the factor is $Y_0$. It
has the stated analytic properties by Theorem~\ref{thm:complex}, and
restricting $\pi^*c_1(X)=0$ to the slice gives $c_1(Y)=0$.
\end{proof}

We now record the vanishing properties of the compact factor needed
for descent. Since $Y$ is a compact simply connected K\"ahler space
with klt singularities, by \cite{CGGN} 
$$q(Y)=\frac12 b_1(Y)=0, \qquad H^0(Y,\Omega_Y^{[1]})=0.$$ The Bochner principle and the resulting
perfect pairing also gives $H^0(Y,\mathscr T_Y)=0$, c.f.,
\cite[Theorem~3.4 and the argument following (4.1.1)]{CGGN}.

Now we claim that the holomorphic isometry group $K$ of $Y$ is finite.
Since $Y$ is
a compact finite-dimensional $\RCD$ space, its metric isometry group
$\operatorname{Isom}(Y)$ is a compact Lie group by
\cite[Theorem~A]{GuijarroSantos}. The subgroup $K$ is closed:
uniform limits in local analytic embeddings are holomorphic, and
the same holds for the inverses. Hence $K$ is itself a compact Lie
group by the closed subgroup theorem. A nontrivial one-parameter subgroup of $K$
would give a nonzero real holomorphic vector field on $Y_{\reg}$ and
its $(1,0)$-part extends by normality, contradicting
vanishing of $H^0(Y,\mathscr T_Y)$. Thus $K$ is discrete and compact, hence
finite.
\subsection{Proof of Theorem~\ref{thm:etale}}
 
The argument is inspired by Cheeger-Gromoll \cite{CG}. Recall that we have holomorphically isometric splitting  $\widetilde X=\R^r\times Y$, with $Y$ compact. It follows from \cite{CG} that 
\[
\operatorname{Isom}(\mathbb R^{2q}\times Y)
\cong
\operatorname{Isom}(\mathbb R^{2q})\times\operatorname{Isom}(Y).
\] 

The kernel of the projection from $\pi_1(X)$ to the finite
holomorphic isometry group $K$ of $Y$ has finite index, and it acts
faithfully, properly and cocompactly on $\C^q$. By Bieberbach's
Theorem \cite{Bieberbach}, it contains a translation lattice
$\Lambda\simeq\Z^{2q}$ of finite index. Since this lattice acts
trivially on $Y$, we obtain
\[
 (\C^q\times Y)/\Lambda=A\times Y,\qquad A=\C^q/\Lambda.
\]
This gives a finite cover of $X$ by local biholomorphisms, which is
therefore \'etale also above the singular points. By finite Riemann
existence \cite[Expos\'e XII, Th\'eor\`eme 5.1]{SGA1}, it is algebraic, and it is
projective since the pullback of $L$ is ample. Its factors are then
projective, so that $A$ is an abelian variety and $Y$ is a projective
Calabi-Yau variety. The metric product descends from the universal
cover, and restriction to the factors, together with the uniqueness
in \cite{EGZ}, identifies the factor metrics with the Ricci-flat
metrics of the restricted polarizations. If $q=0$, the same argument
gives a finite \'etale cover $Y=\tX\to X$.

It remains to identify $q$. By vanishing of $q(Y)$, we have
$q(A\times Y)=q$, so that $q'(X)\ge q$. Conversely, let $X_1\to X$ be
a connected finite \'etale cover. A connected component $X_2$ of its
fiber product with $A\times Y$ is a finite \'etale cover of both.
Since $Y$ is simply connected, $X_2=A'\times Y$ for some isogeny
$A'\to A$, and thus $q(X_2)=q$. Since the trace map makes pullback
on $H^1(\cO)$ injective, we conclude that $q(X_1)\le q$. Taking the
supremum yields $q=q'(X)$. As the lattice $\Lambda$ has finite index
in $\pi_1(X)$, the proof is complete.

\section{The metric and analytic completions}\label{sec:completion}

Let $p:M\to X_{\reg}$ be the universal cover, with deck group
$\Gamma=\pi_1(X_{\reg})$. We equip $M$ with the smooth Ricci-flat
metric $p^*\omega$, its intrinsic length distance $d_M$, and its
Riemannian volume. We denote by $\widehat M$ its metric completion,
use $d_M$ also for the completed distance, and extend the volume
measure by zero on $S:=\widehat M\setminus M$, the resulting measure
is denoted by $\mu_M$. Note that $M$ covers $X_{\reg}$, whereas $\tX$
in Sections~\ref{sec:complex}-\ref{sec:etale} covers the whole space
$X$.

In this section, we identify the normal analytic extension
$\widehat M_{\mathrm{an}}$ with $\widehat M$, and establish its
Sobolev structure. The RCD condition will be proved in
Section~\ref{sec:lifting}.

\subsection{The Grauert-Remmert extension}
By Braun's theorem \cite[Theorem 1]{Braun},   every point of $X$
has a sufficiently small connected neighborhood $U$ such that
$U^\circ=U\cap X_{\reg}$ is connected and $\pi_1(U^\circ)$ is finite.
After a choice of base points, each component of $p^{-1}(U^\circ)$
is the finite Galois cover corresponding to the kernel of
\[
 \pi_1(U^\circ)\longrightarrow\Gamma.
\]
By the Grauert-Remmert theorem
\cite[Expos\'e XII, Th\'eor\`eme 5.4]{SGA1}, this component extends
uniquely to a finite normal map $V\to U$. By uniqueness on overlaps,
these finite extensions glue to a normal analytic space, together
with a map
\[
 \widehat p_{\mathrm{an}}:\widehat M_{\mathrm{an}}\longrightarrow X,
 \qquad \widehat p_{\mathrm{an}}^{-1}(X_{\reg})=M.
\]
see also \cite[Corollary 3]{Braun}. The map is \emph{locally finite},
in the sense that over each such $U$ it is a disjoint union of finite
maps. Since each $V\to U$ is quasi-\'etale, $V$ is klt, and the added
analytic set has complex codimension at least two.

The deck transformations extend uniquely. If $G\subset\Gamma$ denotes
the finite image of $\pi_1(U^\circ)$, the local description is
\begin{equation}\label{eq:local-induced}
 \widehat p_{\mathrm{an}}^{-1}(U)=\Gamma\times_G V,
 \qquad U=V/G.
\end{equation}
Here $\Gamma\times_G V$ denotes the quotient of $\Gamma\times V$
by $(\gamma,v)\sim(\gamma g^{-1},gv)$ for $g\in G$.
In particular, $\widehat M_{\mathrm{an}}/\Gamma=X$ analytically.

\subsection{The two completions are the same}\label{subsec:completion}
Since the projection preserves lengths locally, $p$ is $1$-Lipschitz,
and it extends to a continuous map $\widehat p:\widehat M\to X$.
\begin{proposition} \label{prop:completion}
The identity on $M$ extends uniquely to a $\Gamma$-equivariant
homeomorphism
\begin{equation}\label{eq:completion-map}
 \iota:\widehat M_{\mathrm{an}}\longrightarrow\widehat M,
 \qquad \widehat p\circ\iota=\widehat p_{\mathrm{an}}.
\end{equation}
Consequently, $\widehat M$ carries the normal klt structure of the
analytic extension, with $M\subset(\widehat M)_{\reg}$, and $S$ is
analytic of complex codimension at least two. Moreover, it is a
proper geodesic metric space. The extended deck action is proper and
cocompact with finite stabilizers, and acts by holomorphic
isometries. In particular, its metric quotient is $(X,d)$.
\end{proposition}
\begin{proof} The proof is straightforward using \cite{GSS}, which establish the identification of the analytic space $X$ and the metric completion of $(X^{\circ},\omega)$. We divide the proof into several steps.

\emph{Step 1:} Fix a finite chart $V\to U$ of degree $m$, with
$B_{4r}(x)\Subset U$. This is possible by \cite{GSS} and $x$ can be a singular point in $U$. Choose $b\in B_r(x)\cap X_{\reg}$, and write
$p^{-1}(b)\cap V=\{b_1,\ldots,b_m\}$. For
$z\in B_r(x)\cap X_{\reg}$, the definition one may choose a curve $\gamma_{z,b}$ connecting $z$ and $b$ such that 
\[
 \gamma_{z,b}\subset X_{\reg},
 \qquad\ell(\gamma_{z,b})<3r,\qquad\gamma_{z,b}\subset B_{4r}(x).
\]
Lifting these curves, we obtain
\begin{equation}\label{eq:finite-ball-cover}
 V\cap M\cap\widehat p_{\mathrm{an}}^{-1}(B_r(x))
 \subset\bigcup_{a=1}^m B_{3r}^{M}(b_a).
\end{equation}Thus if $A$ is a connected subset  of the left-hand side, then
\begin{equation}\label{eq:diameter}
 \operatorname{diam}_{d_M}A\le6mr.
\end{equation}
For $y\in V$ over $x$, take connected normal neighborhoods
$W_r\subset V\cap\widehat p_{\mathrm{an}}^{-1}(B_r(x))$ of $y$.
By normality, $W_r\cap M$ is connected, and hence
\[
 \operatorname{diam}_{d_M}(W_r\cap M)\le6mr\longrightarrow0.
\]
Thus analytic convergence to $y$ determines a unique metric limit,
which defines the continuous extension $\iota$ in
\eqref{eq:completion-map}.

\emph{Step 2:} We show that the extended map $\iota$ is injective. To this end, fix a point $x\in U$ and locally finite map  $V\rightarrow U$ as in step 1. Also choose disjoint
neighborhoods $W_a$ of the finite fiber. By properness, there is
$U'\Subset U$ with
$V\cap\widehat p_{\mathrm{an}}^{-1}(U')\subset\bigcup_a W_a$.
Choose $B_{2\delta}(x)\subset U'$ and set
$A_a=W_a\cap M\cap p^{-1}(B_\delta(x))$. Any curve joining different
sets $A_a$ must project outside $B_{2\delta}(x)$, whence
\[
 d_M(A_a,A_b)\ge\delta\qquad(a\ne b).
\]
Different components over $U$ are separated by the same argument.
Thus $\iota$ is injective.

\emph{Step 3:} Next, let $(y_j)\subset M$ be a metric Cauchy sequence. By definition, $p(y_j)$ is also a metric Cauchy sequence converging to some point $x$ downstairs. So, for $j$ sufficiently large, $y_j$ are contained in a
finite chart $V\to U$.  Therefore, after passing to a subsequence,
\begin{equation}\label{eqn:uniy}
 y_{j_k}\longrightarrow y\text{ analytically},\qquad
 d_M(y_{j_k},\iota(y))\longrightarrow0.
\end{equation}
This proves surjectivity of $\iota$. Moreover, by the
injectivity of $\iota$, such kind of $y$ in \eqref{eqn:uniy} is unique and hence $\iota^{-1}$ is also continuous. This proves that $\iota$ is a homeomorphism.

\emph{Step 4:} Finally, we show that the metric quotient of $\hat M$ by $\gamma$ is isometric to $X$. we identify the two spaces via $\iota$. The completion of a
length space is a length space. Completeness and local compactness
therefore make $\widehat M$ proper and geodesic. The finite local
model \eqref{eq:local-induced} shows that the action is holomorphic
and proper with finite stabilizers, while compact lifts of a finite
cover of $X$ give cocompactness. The action is isometric by density
of $M$. Finally, we show that
\begin{equation}\label{eq:metric-quotient}
 d(\widehat p(y),\widehat p(z))
   =\inf_{\gamma\in\Gamma}d_M(y,\gamma z)=:\delta(z,y)
 \qquad(y,z\in\widehat M).
\end{equation}
By curve lifting, it is clear that the left hand side is less than the right hand side and if $(z,y)$ is a pair of points in $M$, then the identity holds. In general, when $(z,y)$ is in $\hat M$, the identity follows from the continuity of $\delta(z,y)$ and approximating $(z,y)$ by points in $M$.
\end{proof}

The smooth locus used below is $M$, note that it may be smaller than
the analytic regular locus $(\widehat M)_{\reg}$.

\subsection{Measure, capacity and the Sobolev space}
For a chart $V\to U$ of degree $m$, we have
\[
 \mu_M(V)=m\mu(U)<\infty.
\]
The measure has full support and is $\Gamma$-invariant.

We use the metric Sobolev space defined by relaxation of the local
Lipschitz slope, as in \cite[Definition 2.3]{HS}. Explicitly, set
\[
 \operatorname{Ch}(u)=\inf_{u_j\to u\text{ in }L^2}
   \liminf_{j\to\infty}\frac12
       \int_{\widehat M}(\operatorname{lip}u_j)^2\,\dd\mu_M,
\]
where the infimum is taken over
$u_j\in\Lip(\widehat M)\cap L^\infty\cap L^2$, and
$\operatorname{lip}u_j(x)=\limsup_{y\to x}
|u_j(y)-u_j(x)|/d_M(x,y)$. Then
\[
 \begin{gathered}
 W^{1,2}(\widehat M,d_M,\mu_M)
   =\{u\in L^2(\mu_M):\operatorname{Ch}(u)<\infty\},\\
 \|u\|_{W^{1,2}}^2=\|u\|_2^2+2\operatorname{Ch}(u).
 \end{gathered}
\]
This definition does not require an RCD hypothesis; the identification
with the smooth finite-energy space is proved below.

By \cite[Lemma 8.2]{GPSS}, we may choose
$\chi_j\in C_c^\infty(X_{\reg})$, eventually equal to one on
each compact subset, such that
\begin{equation}\label{eq:base-capacity}
 0\le\chi_j\le1,\qquad \chi_j\longrightarrow1
 \text{ locally on }X_{\reg},\qquad
 \int_X|\nabla\chi_j|^2\,\dd\mu\longrightarrow0.
\end{equation}
For $\eta\in\Lip_c(V)$, set $a_j=\eta(1-\chi_j\circ\widehat p)$.
Then
\[
 \int_M(a_j^2+|\nabla a_j|^2)\,\dd\mu_M
 \le C_\eta m\int_U\big((1-\chi_j)^2+|\nabla\chi_j|^2\big)\,\dd\mu
 \longrightarrow0.
\]
Taking $\eta=1$ near a prescribed compact subset of $S\cap V$,
and using finitely many charts, we obtain
\begin{equation}\label{eq:local-capacity}
 S\text{ has zero local }W^{1,2}\text{-capacity in }\widehat M.
\end{equation}
Here the functions $a_j$ are Lipschitz, so their Cheeger energies are
bounded by the integrals over the smooth locus displayed above.

\begin{lemma} \label{lem:form}
The Sobolev space
$W^{1,2}(\widehat M,d_M,\mu_M)$ coincides with 
\[
 \left\{u\in W^{1,2}_{\loc}(M):
       \int_M\bigl(|u|^2+|\nabla u|^2\bigr)\,\dd\mu_M<\infty\right\}.
\]

\end{lemma}
\begin{proof}
We apply the cutoff and density argument of Guo-Phong-Song-Sturm
\cite[Lemma 8.2 and Proposition 8.1]{GPSS}, using
\eqref{eq:local-capacity} together with distance cutoffs at infinity
on the completed cover. Relaxation on the smooth locus gives the
energy identity and identifies the weak gradients. If
$|\nabla u|\le1$, the local Sobolev representative on $M$ is
locally Lipschitz. Integration along curves in $M$ then gives
$|u(x)-u(y)|\le d_M(x,y)$, and the representative extends to
$\widehat M$ with the same Lipschitz constant.
\end{proof}

Under this identification, the weak gradient coincides almost
everywhere with the weak gradient on the smooth locus, and
\begin{equation}\label{eq:form-core}
 W^{1,2}(\widehat M)=\overline{C_c^\infty(M)}^{\|\cdot\|_{1,2}},
 \qquad 2\operatorname{Ch}(u)=\int_M|\nabla u|^2\,\dd\mu_M,
\end{equation}
where $\|u\|_{1,2}^2=\int_M(|u|^2+|\nabla u|^2)\,\dd\mu_M$.
In particular, the Cheeger energy is quadratic, and $\widehat M$
has the Sobolev-to-Lipschitz property.

We write $\Delta$ for the nonpositive self-adjoint Laplacian
associated with this form: for $u\in D(\Delta)$,
\[
 \int_{\widehat M}\langle\nabla u,\nabla\varphi\rangle\,\dd\mu_M
 =-\int_{\widehat M}(\Delta u)\varphi\,\dd\mu_M
 \qquad\text{for every }\varphi\in W^{1,2}(\widehat M).
\]
Its restriction to $M$ is the usual operator
$\operatorname{div}\nabla$. For every $v\in L^2$, the equation
$f-\Delta f=v$ has a unique solution in $W^{1,2}$, which we denote
by $(1-\Delta)^{-1}v$. We emphasize that this notation is defined
before the RCD conclusion is established.

For later use, we note that the same cutoff argument on the compact
base identifies $W^{1,2}(X)$ with the space of finite-energy
functions on $X_{\reg}$. Moreover, the Sobolev inequality of
\cite[Theorem 8.1]{GPSS} provides constants $q_S>1$ and $C_S$ such
that
\begin{equation}\label{eq:base-sobolev}
 \|u\|_{L^{2q_S}(X)}^2
       \le C_S\int_X(u^2+|\nabla u|^2)\,\dd\mu.
\end{equation}

\section{A compact orbifold resolution and Ricci entropy}
\label{sec:approximation}

In this section, we use the analytic extension and the deck action of
Proposition~\ref{prop:completion} to construct a compact orbifold on
which the Ricci entropy estimates can be integrated. As we shall see
in Section~\ref{sec:lifting}, the lifted metrics approximate
$p^*\omega$ on $M$. For $n=1$, $X$ is an elliptic curve and the
completed cover is the flat complex line. We shall therefore assume
$n\ge2$ throughout this section, as well as in the analytic estimates
of Section~\ref{sec:lifting}.

\subsection{Equivariant resolution}
Functorial resolution and principalization of the pair
$(\widehat M,S)$ \cite{BM,BM97} provide a $\Gamma$-equivariant projective
modification
\[
 r:Z\longrightarrow\widehat M,
 \qquad D_Z=Z\setminus M,
\]
where $Z$ is smooth, $r$ is an isomorphism over $M$, and $D_Z$ has
simple normal crossings. Hence
\[
 \Q=[Z/\Gamma],\qquad D=D_Z/\Gamma,
\]
is a compact smooth complex orbifold, equipped with a map $q:\Q\to X$
which identifies $\Q\setminus D$ with $X_{\reg}$.
The blowups provide a $\Gamma$-invariant effective exceptional
divisor $D_0$ with support $D_Z$, such that $\mathcal O(-D_0)$ is
relatively ample over $\widehat M$. On each finite chart it is also
relatively ample over $X$. The order of stabilizer of $\Gamma$ action on $Z$ are bounded, so a
divisible power of this naturally linearized line bundle descends
to a relatively ample line bundle on the coarse quotient. The latter
is therefore projective over $X$, hence projective. Gluing invariant
strictly plurisubharmonic local potentials by regularized maxima
gives positive orbifold curvature in the pullback of an ample class,
as in \cite[Section 3.4]{EGZ}. Fix an orbifold K\"ahler form $\beta$ on $\Q$.

We choose a Fubini-Study form $\theta\in c_1(L)$ and set
$\alpha=q^*\theta$. For the global orbifold divisors $D_i$, choose
defining sections $s_i$ and smooth orbifold Hermitian metrics $h_i$.
The exceptional divisor construction gives numbers $b_i>0$ such that,
with
\[
 t_i=|s_i|_{h_i}^2\le e^{-2},\qquad L_i=-\log t_i,
 \qquad\psi_D=\sum_i b_i\log t_i,
\]
we have
\begin{equation}\label{eq:barrier}
 \alpha+\ddc\psi_D\ge c\beta\quad\text{on }\Q\setminus D.
\end{equation}
Indeed, a small exceptional correction to $[\alpha]$ is K\"ahler by
relative ampleness: represent that class by positive curvature and
choose the $h_i$ accordingly, then apply the Poincar\'e-Lelong
formula. This is the construction in \cite[Section 4]{GSS}.
The sections, norms and regularizations below are thus compatible
on overlaps of uniformizing charts.

\subsection{The adapted volume and its regularization}
Let $\Omega_X$ be the adapted Calabi-Yau volume form, normalized so
that $\omega^n=\Omega_X$. It may be constructed from a trivialization
of a sufficiently divisible multiple of $K_X$. In every finite
chart $V\to U$, the cover is quasi-\'etale, so that
$K_V=\widehat p^*K_U$. On the orbifold, the adapted volume takes the
form
\begin{equation}\label{eq:adapted}
 q^*\Omega_X=\prod_i t_i^{a_i}\Omega_0,
 \qquad a_i>-1,
\end{equation}
where $\Omega_0$ is a smooth positive orbifold volume form.
Indeed, on a uniformizing resolution chart $r_V:Z_V\to V$, the
formula $K_{Z_V}=r_V^*K_V+\sum_i a_iD_i$ gives the factors
$t_i^{a_i}=|s_i|_{h_i}^{2a_i}$, after taking the appropriate root of
the squared norm of a pluricanonical generator. Since finite
quasi-\'etale covers preserve klt singularities, $V$ is klt and its
discrepancies satisfy $a_i>-1$. These discrepancies are computed on
the uniformizing charts, with all divisorial stabilizers retained.

For $0<\eps\le1/4$, we set $\alpha_\eps=\alpha+\eps\beta$ and solve
\begin{equation}\label{eq:approx-ma}
 \omega_\eps=\alpha_\eps+\ddc u_\eps>0,
 \qquad
 \omega_\eps^n=c_\eps\prod_i(t_i+\eps)^{a_i}\Omega_0,
 \qquad\sup_\Q u_\eps=0.
\end{equation}
Here $c_\eps$ is the normalizing constant. The existence of a
solution follows from the orbifold version of the Calabi conjecture,
obtained by the usual continuity method in finite uniformizing
charts; see \cite[Theorem 4.3]{CC}. Equation~\eqref{eq:approx-ma} is
the prescribed volume approximation which will be used throughout
this section.

\begin{lemma} \label{lem:approx}
There exist $p>1$ and constants $C,N$ independent of $\eps$ such that
\begin{equation}\label{eq:approx-bounds}
 \|u_\eps\|_\infty\le C,
 \quad
 \left\|\frac{\omega_\eps^n}{\beta^n}\right\|_{L^p(\beta^n)}\le C,
 \quad
 \tr_{\omega_\eps}\beta\le C\prod_i t_i^{-N}
 \quad\text{off }D.
\end{equation}
Moreover, $c_\eps\to1$, and $\omega_\eps\to q^*\omega$ locally
smoothly on $\Q\setminus D=X_{\reg}$.
\end{lemma}
\begin{proof}
Choose $p>1$ such that $pa_i>-1$ whenever $a_i<0$. In normal crossing
coordinates, we find
\[
 \sup_\eps\left\|\prod_i(t_i+\eps)^{a_i}\right\|_{L^p(\Omega_0)}<\infty,
 \qquad
 \prod_i(t_i+\eps)^{a_i}\longrightarrow\prod_i t_i^{a_i}
 \text{ in }L^1(\Omega_0).
\]
Hence
\[
 c_\eps=
 \frac{\int_\Q\alpha_\eps^n}
      {\int_\Q\prod_i(t_i+\eps)^{a_i}\Omega_0}
 \longrightarrow1.
\]
The bounded-potential estimates of \cite{EGZ} and
\cite[Theorem 2.2(A), (C)]{DP} apply to the fixed
semipositive big form $\alpha$ and its perturbations
$\alpha+\eps\beta$. The same capacity proof works in finite
uniformizing charts, with integrals divided by the chart group
orders, and gives a bound uniform in $\eps$.
Using \eqref{eq:barrier}, the prescribed-volume trace estimate of
\cite[Lemma 7.1]{GPSS}, also used in \cite[Lemma 4.3]{GSS},
and the determinant bound give
\[
 \begin{gathered}
 \tr_\beta\omega_\eps\le C\prod_i t_i^{-N},\qquad
 \frac{\omega_\eps^n}{\beta^n}\ge c\prod_{a_i>0}t_i^{a_i},\\
 \tr_{\omega_\eps}\beta
 \le C\frac{(\tr_\beta\omega_\eps)^{n-1}}{\omega_\eps^n/\beta^n}.
 \end{gathered}
\]
Increasing $N$ if necessary, we obtain \eqref{eq:approx-bounds}.
These local estimates are invariant under the finite chart groups.
Interior elliptic estimates give smooth subsequential convergence
away from $D$, and the uniqueness of bounded solutions of the
Monge-Amp\`ere equation \cite{EGZ} identifies every limit with
$q^*\omega$.
\end{proof}

\subsection{The entropy estimate}
Let $\rho_\eps$ denote the magnitude of the negative part of the
smallest Ricci eigenvalue of $\omega_\eps$.

\begin{proposition} \label{prop:entropy}
The metrics \eqref{eq:approx-ma} satisfy
\begin{equation}\label{eq:entropy}
 \sup_{0<\eps\le1/4}
 \int_\Q\rho_\eps\log(1+\rho_\eps)
                 \frac{\omega_\eps^n}{n!}<\infty.
\end{equation}
Their lifts to $Z$ are complete smooth $\Gamma$-invariant
K\"ahler metrics, which converge locally smoothly to $p^*\omega$
on $M$.
\end{proposition}
\begin{proof}
We follow \cite[Lemmas 6.1--6.4 and Proposition 6.1]{GSS}. Writing
$\ddc\log t_i=2\pi[D_i]-\Theta_i$, we set
\[
 \begin{aligned}
 P_{i,\eps}
 &=\ddc\log(t_i+\eps)+\frac{t_i}{t_i+\eps}\Theta_i\\
 &=\frac{\eps\sqrt{-1}\langle D s_i\wedge\overline{D s_i}\rangle}
 {(t_i+\eps)^2},\qquad 0\le P_{i,\eps}\le Ct_i^{-1}\beta
 \end{aligned}
\]
off $D_i$. Equation~\eqref{eq:approx-ma} then gives
\[
 \begin{gathered}
 \Ric(\omega_\eps)\ge-C\beta-\sum_{a_i>0}a_iP_{i,\eps},\\
 \rho_\eps\le R_\eps:=1+C\tr_{\omega_\eps}\beta
             +\sum_{a_i>0}a_i\tr_{\omega_\eps}P_{i,\eps}.
 \end{gathered}
\]
By Lemma~\ref{lem:approx},
\begin{equation}\label{eq:log-R}
 \log(1+R_\eps)\le C\Big(1+\sum_k L_k\Big)
 \quad\text{almost everywhere}.
\end{equation}

Next, set $T_\eps=\sum_{j=0}^{n-2}\omega_\eps^j\wedge
\alpha_\eps^{n-2-j}$. The orbifold Stokes theorem and the
Poincar\'e-Lelong formula give
\begin{equation}\label{eq:log-moment}
 \begin{aligned}
 \int_\Q L_k\beta\wedge\omega_\eps^{n-1}
 &=\int_\Q L_k\beta\wedge\alpha_\eps^{n-1}\\
 &\quad+\int_\Q u_\eps\beta\wedge(\Theta_k-2\pi[D_k])\wedge T_\eps\\
 &\le C(1+\|u_\eps\|_\infty)\le C.
 \end{aligned}
\end{equation}
The last bound uses uniform bounds on the cohomological
integrals. Since $|\log(t_i+\eps)|\le L_i$, a second integration by
parts yields
\begin{equation}\label{eq:spike-moment}
 \begin{aligned}
 \int_\Q L_kP_{i,\eps}\wedge\omega_\eps^{n-1}
 &\le C-2\pi\int_{D_k}\log(t_i+\eps)\omega_\eps^{n-1}\\
 &\le C+2\pi|\log\eps|\int_{D_k}\alpha_\eps^{n-1}.
 \end{aligned}
\end{equation}
Here the logarithm is smooth for fixed $\eps>0$, and the same
orbifold multiplicities occur in the volume and divisor integrals.
Since $d_k:=\dim q(D_k)\le n-2$, we have
\begin{equation}\label{eq:exceptional-mass}
 \int_{D_k}\alpha_\eps^{n-1}
 =\sum_{j=0}^{d_k}\binom{n-1}{j}\eps^{n-1-j}
       \int_{D_k}\alpha^j\wedge\beta^{n-1-j}
 \le C\eps.
\end{equation}
The trace-wedge identity, the estimates
\eqref{eq:log-R}-\eqref{eq:exceptional-mass}, and the $L^p$ bound on
the density now give
\[
 \int_\Q\rho_\eps\log(1+\rho_\eps)\omega_\eps^n
 \le\int_\Q R_\eps\log(1+R_\eps)\omega_\eps^n
 \le C(1+\eps|\log\eps|).
\]
The lifted metrics are complete because $\Q$ is compact, and their
local smooth convergence follows from Lemma~\ref{lem:approx}.
\end{proof}

\section{The RCD structure}\label{sec:lifting}

We now prove the curvature statement for the completion constructed
in Section~\ref{sec:completion}.

\begin{theorem}\label{thm:lifting}
Let $X$ be a projective Calabi-Yau variety of complex dimension $n$,
equipped with the Ricci-flat metric of a fixed polarization. Then the
metric completion of the universal cover $M\to X_{\reg}$ satisfies
\begin{equation}\label{eq:main}
 (\widehat M,d_M,\mu_M)\text{ is }\RCD(0,2n).
\end{equation}
\end{theorem}

The analytic identification and the proper cocompact deck action
have already been obtained in Proposition~\ref{prop:completion}.
It remains to prove global gradient bounds for the solutions of
$f-\Delta f=v$, using the approximations of
Section~\ref{sec:approximation}. The cutoff argument of Honda-Sun
will then give Theorem~\ref{thm:lifting}.

\subsection{Elliptic equations, convergence and the gradient bound}

For each fixed $\eps>0$, the metric $\omega_\eps$ is smooth in the
finite uniformizing charts of the compact orbifold $\Q$.
The differential calculations below are therefore the usual smooth
ones, and their integrals are justified by the orbifold Stokes
theorem. We use the calculations of \cite{GSS}, summing the scalar
equations over the covering sheets.

We write $g_\eps$ for the metric on $\Q$ as well as for its lift to
$Z$, and $\mu_\eps$, $\widetilde\mu_\eps$ for their volume measures.
The target measures are $\mu$ on $X$ and $\mu_M$ on $\widehat M$.
For a real-valued function $f$ on $Z$, we use the notation
$F=(F_\gamma)_{\gamma\in\Gamma}$ for its translates in a local
uniformizing chart, with
\begin{equation}\label{eq:pack}
 F_\gamma(z)=f(\gamma z),\qquad
 |F|^2=\sum_{\gamma\in\Gamma}F_\gamma^2,\qquad
 |\nabla F|^2=\sum_{\gamma\in\Gamma}|\nabla F_\gamma|^2.
\end{equation}
All derivatives are taken componentwise, and inner products are
summed over the same index set. Since a change of local lift permutes
the indices, these scalar sums are well defined on $\Q$.
For $f$ of finite $W^{1,2}$ norm, unfolding gives
\begin{equation}\label{eq:unfold}
\begin{aligned}
 \int_\Q|F|^2\,\dd\mu_\eps
   &=\int_Z|f|^2\,\dd\widetilde\mu_\eps,\\
 \int_\Q|\nabla F|^2\,\dd\mu_\eps
   &=\int_Z|\nabla f|^2\,\dd\widetilde\mu_\eps.
\end{aligned}
\end{equation}
The action is free away from the null divisor $D_Z$, and these
identities also hold for the target metric on $M\to X_{\reg}$.

Fix $v\in C_c^\infty(M;\R)$, extend it by zero to $Z$, and write
$V=(v(\gamma z))_{\gamma\in\Gamma}$. On each relatively compact
chart, only finitely many components are nonzero, and all of them
vanish near $D$. By properness, the number of points of each
$\Gamma$-orbit in $\operatorname{supp}v$ is uniformly bounded, so
local smooth convergence gives
\begin{equation}\label{eq:source-bounds}
 \sup_\eps\big(\|V\|_\infty+
       \|\nabla V\|_{\infty,g_\eps}+\|V\|_{L^2(\mu_\eps)}\big)<\infty.
\end{equation}
On the complete manifold $(Z,g_\eps)$, let $f_\eps$ be the unique
finite-energy solution in $L^2$ of
\begin{equation}\label{eq:elliptic-eps}
 (1-\Delta_\eps)f_\eps=v.
\end{equation}
Each component then satisfies
$(1-\Delta_\eps)F_{\eps,\gamma}=V_\gamma$.
Testing the equation and using \eqref{eq:unfold}, we obtain
\begin{equation}\label{eq:elliptic-bound}
 \int_\Q\big(|F_\eps|^2+|\nabla F_\eps|^2
                     +|\Delta_\eps F_\eps|^2\big)\,\dd\mu_\eps\le C_v.
\end{equation}
For fixed $\eps$, the scalar interior elliptic estimates, summed over
the components, give, on nested uniformizing charts $U'\Subset U$
and for every $k\ge0$,
\[
 \begin{aligned}
 \sum_\gamma\|F_{\eps,\gamma}\|_{C^k(U')}^2
 &\le C_{\eps,k,U}\sum_\gamma
                \|F_{\eps,\gamma}\|_{L^2(U)}^2\\
 &\quad+C_{\eps,k,U}\sum_\gamma
                \|V_\gamma\|_{C^{k+1}(U)}^2<\infty.
 \end{aligned}
\]
Thus all the squared-derivative sums and pairings below converge
locally smoothly. They can be differentiated using finite partial
sums, and the resulting full sums are invariant under changes of
chart.

We write $B_F=(B_{F_\gamma})$ for the componentwise $J$-invariant
Hessian,
\[
 \begin{gathered}
 B_{F_\gamma}(\xi,\eta)=\tfrac12
 \big(\Hess F_\gamma(\xi,\eta)+\Hess F_\gamma(J\xi,J\eta)\big),\\
 |B_F|^2=\sum_\gamma|B_{F_\gamma}|^2.
 \end{gathered}
\]
We also set $h_\eps=\log(1+|\nabla F_\eps|^2)$ and $B_\eps=B_{F_\eps}$.

\begin{lemma} \label{lem:log-estimate}
There is a constant $C_v$, independent of $\eps$, such that
\begin{equation}\label{eq:uniform-seeds}
 \int_\Q\big(h_\eps^2+|\nabla h_\eps|^2+|B_\eps|^2\big)
                                         \,\dd\mu_\eps\le C_v.
\end{equation}
\end{lemma}
\begin{proof}
The orbifold Stokes theorem, applied to the global form
$\sum_\gamma d^cF_{\eps,\gamma}\wedge dd^cF_{\eps,\gamma}
\wedge\omega_\eps^{n-2}$, gives the summed scalar identity of
\cite[Lemma 5.1, equation (5.2)]{GSS}:
\[
 \int_\Q|B_\eps|^2\,\dd\mu_\eps
 =\tfrac12\int_\Q|\Delta_\eps F_\eps|^2\,\dd\mu_\eps\le C_v.
\]
The local convergence established above justifies the infinite sum
and the use of Stokes' theorem. In dimension one, the identity holds
pointwise.

We first consider a finite collection of components. Set
$w=|\nabla F|$, let $B_a^\perp$ be the $J$-anti-invariant Hessian of
$F_a$, and put $T((u_a))=\sum_a B_a^\perp u_a$. Then
\[
 TT^*=\sum_a(B_a^\perp)^2,\qquad [TT^*,J]=0,\qquad
 \|T\|_{\mathrm{op}}^2\le\tfrac12\sum_a|B_a^\perp|^2.
\]
Taking $u_a=\nabla F_a/w$ where $w>0$, we obtain
\[
 |\nabla w|\le\tfrac1{\sqrt2}
       \Big(\sum_a|B_a^\perp|^2\Big)^{1/2}+|B_F|.
\]
Since this bound is independent of the number of components, it
passes to the full sum. This is the Kato estimate used in
\cite[Lemma 3.1, equation (3.2)]{GSS}; on $\{w=0\}$, we have
$\nabla w=0$ almost everywhere. Since
$|\Hess F|^2=\sum_a|B_a^\perp|^2+|B_F|^2$, this implies
\[
 |\Hess F|^2\ge2|\nabla w|^2-4|B_F||\nabla w|.
\]

For $(1-\Delta)F=V$ and $|\nabla V|\le A$, the Bochner formula gives
\[
 \begin{gathered}
 \tfrac12\Delta w^2=|\Hess F|^2+\Ric(\nabla F,\nabla F)
                     +w^2-\langle\nabla F,\nabla V\rangle,\\
 \frac{w^2-\langle\nabla F,\nabla V\rangle}{1+w^2}\ge-\frac A2.
 \end{gathered}
\]
For $h=\log(1+w^2)$, substitute these inequalities into
\[
 \Delta h=\frac{\Delta w^2}{1+w^2}
             -\frac{4w^2|\nabla w|^2}{(1+w^2)^2}.
\]
On $\{w\le1\}$ use Young's inequality and on $\{w\ge1\}$ use
$|\nabla h|=2w|\nabla w|/(1+w^2)$. With $\Ric\ge-\rho g$, this gives
the calculation of \cite[Lemma 3.2, equations (3.5)-(3.6)]{GSS}
for the present inhomogeneous equation:
\begin{equation}\label{eq:log-drift}
 \Delta h\ge-C_A(1+\rho)
       -Ce^{-h/2}|B_F||\nabla h|-Ce^{-h}|B_F|^2.
\end{equation}
As in \cite[Lemma 5.2, equations (5.3)-(5.5)]{GSS}, Young's
inequality then gives
\[
 \begin{gathered}
 \Delta(1+h)^2\ge|\nabla h|^2-C|B_F|^2-C_A(1+h)(1+\rho),\\
 \rho h\le\rho\log(1+\rho)+e^h.
 \end{gathered}
\]
We integrate over $\Q$, using $h^2\le Ce^h$ and
\[
 \int_\Q e^{h_\eps}\,\dd\mu_\eps
 =\mu_\eps(\Q)+\int_\Q|\nabla F_\eps|^2\,\dd\mu_\eps\le C_v.
\]
Proposition~\ref{prop:entropy} and \eqref{eq:elliptic-bound} then
imply
\[
 \int_\Q(h_\eps^2+|\nabla h_\eps|^2+|B_\eps|^2)\,\dd\mu_\eps
 \le C_v\left(1+\int_\Q\rho_\eps\log(1+\rho_\eps)\,\dd\mu_\eps\right)
 \le C_v.
\]
\end{proof}

\subsubsection{Convergence without loss of energy}
\begin{lemma} \label{lem:solution-convergence}
For $v\in C_c^\infty(M)$, the solutions of
\eqref{eq:elliptic-eps} converge locally smoothly on $M$ to the
solution $f=(1-\Delta)^{-1}v$ on $\widehat M$. Moreover,
\begin{align}
 \int_Z f_\eps^2\,\dd\widetilde\mu_\eps
     &\longrightarrow\int_M f^2\,\dd\mu_M,
       \label{eq:L2-convergence}\\
 \int_Z|\nabla f_\eps|_{g_\eps}^2\,\dd\widetilde\mu_\eps
     &\longrightarrow\int_M|\nabla f|^2\,\dd\mu_M.
       \label{eq:energy-convergence}
\end{align}
\end{lemma}
\begin{proof}
Interior elliptic estimates and \eqref{eq:elliptic-bound} give
$f_\eps\to f$ in $C^\infty_{\loc}(M)$ along a subsequence, with
\[
 \int_M(f^2+|\nabla f|^2)\,\dd\mu_M
 \le\liminf_{\eps\to0}\int_Z(f_\eps^2+|\nabla f_\eps|^2)
                               \,\dd\widetilde\mu_\eps<\infty.
\]
Testing against $C_c^\infty(M)$ and using Lemma~\ref{lem:form},
we see that $(1-\Delta)f=v$ on the full form domain. By uniqueness,
the convergence holds without passing to a subsequence.

Since $v$ has compact support in $M$, we have
\begin{equation}\label{eq:energy-identity}
 \begin{split}
 \int_Z(f_\eps^2+|\nabla f_\eps|^2)\,\dd\widetilde\mu_\eps
 &=\int_Z vf_\eps\,\dd\widetilde\mu_\eps\\
 &\longrightarrow\int_M vf\,\dd\mu_M
 =\int_M(f^2+|\nabla f|^2)\,\dd\mu_M.
 \end{split}
\end{equation}
By Fatou's lemma, the corresponding lower bound holds for each
nonnegative term separately. Since the sums are equal, this proves
\eqref{eq:L2-convergence} and \eqref{eq:energy-convergence}.
\end{proof}

On $X_{\reg}$, let $F=(F_\gamma)$ denote the local branches of $f$,
as in \eqref{eq:pack}. Each branch solves
$(1-\Delta)F_\gamma=V_\gamma$, and the same summed scalar interior
estimates show that the derivative sums are smooth there. We define
\begin{equation}\label{eq:target-sums}
 s=|\nabla F|^2,\qquad h=\log(1+s).
\end{equation}
Here $B_F$ is the $J$-invariant Hessian defined above. Note that
both $s$ and $|B_F|^2$ are sums over all the sheets of $M$.

\begin{lemma} \label{lem:target-seeds}
The quantities \eqref{eq:target-sums} satisfy
\begin{equation}\label{eq:target-seeds}
 \int_X(e^h+h^2+|\nabla h|^2+|B_F|^2)\,\dd\mu<\infty,
 \qquad h\in W^{1,2}(X).
\end{equation}
\end{lemma}
\begin{proof}
By unfolding and Fatou's lemma,
\[
 \int_X(e^h-1)\,\dd\mu=\int_M|\nabla f|^2\,\dd\mu_M,
 \qquad \int_X|B_F|^2\,\dd\mu\le C_v.
\]
For the second inequality, one first sums over finitely many sheets
on compact regular charts, and then exhausts $M$.

Given $\delta>0$, choose $0\le\eta\le1$ in $C_c^\infty(M)$ with
$\int_M(1-\eta)|\nabla f|^2\,\dd\mu_M<\delta$.
By \eqref{eq:energy-convergence} and local smooth convergence,
\[
 \lim_{\eps\to0}\int_Z(1-\eta)|\nabla f_\eps|_{g_\eps}^2
                                  \,\dd\widetilde\mu_\eps
 =\int_M(1-\eta)|\nabla f|^2\,\dd\mu_M<\delta.
\]
Over $U\Subset X_{\reg}$, the $\eta$-weighted sheet sum is locally
finite and converges smoothly, while unfolding bounds the remaining
terms:
\[
 \limsup_{\eps\to0}
 \big\||\nabla F_\eps|_{g_\eps}^2-s\big\|_{L^1(U,\mu)}
 \le C_U\delta.
\]
Letting $\delta\downarrow0$, we obtain
\begin{equation}\label{eq:s-strong}
 |\nabla F_\eps|_{g_\eps}^2\longrightarrow s
 \quad\text{in }L^1_{\loc}(X_{\reg},\mu).
\end{equation}
Since $|\log(1+a)-\log(1+c)|^2\le C|a-c|$ for $a,c\ge0$, it follows
that
\[
 h_\eps\longrightarrow h\text{ in }L^2(U),\qquad
 \nabla h_\eps\rightharpoonup\nabla h\text{ in }L^2(U),
\]
where \eqref{eq:uniform-seeds} supplies the bound on the weak
gradients. Exhaustion and lower semicontinuity then yield
$\int_X(h^2+|\nabla h|^2)\,\dd\mu\le C_v$.
Finally, \eqref{eq:base-capacity} and \cite[Lemma 2.2]{GSS}
extend $h$ to a function in $W^{1,2}(X)$.
\end{proof}

\subsubsection{Gradient estimates}

\begin{proposition} \label{prop:gradient}
For every $v\in C_c^\infty(M)$, the function
$(1-\Delta)^{-1}v$ is globally Lipschitz on $\widehat M$.
\end{proposition}
\begin{proof}
Let $f=(1-\Delta)^{-1}v$, and let $F=(F_\gamma)$ denote its local
branches. The sum over the sheets is
\[
 s(x)=|\nabla F(x)|^2
     =\sum_{y\in p^{-1}(x)}|\nabla f(y)|^2,
 \qquad z(x)=e^{h(x)/2}=\sqrt{1+s(x)}.
\]
Since a change of local lift permutes the sheets, $s$ and $z$
are well-defined scalar functions on $X_{\reg}$. By unfolding,
\[
 \|z\|_{L^2(X)}^2
 =\mu(X)+\int_M|\nabla f|^2\,\dd\mu_M<\infty.
\]
All integrals below are taken over the compact base $X$.
We set $A_v=\|\nabla V\|_\infty$. By Ricci-flatness and
\eqref{eq:log-drift}, we have
\begin{equation}\label{eq:target-drift}
 \Delta h\ge-C_v-Ce^{-h/2}|B_F||\nabla h|-Ce^{-h}|B_F|^2.
\end{equation}
As in \cite[Lemma 3.3, equations (3.9)-(3.13)]{GSS}, we use
\[
 \Phi_R(h)=\frac{e^h}{1+e^h/R},\qquad
 \Phi_R'(h)=\frac{e^h}{(1+e^h/R)^2},\qquad
 \Phi_R^2=e^h\Phi_R'.
\]
Testing with $\chi_j\Phi_R(h)$ and applying Young's inequality gives
\begin{equation}\label{eq:weighted-h}
 \frac12\int_X\chi_j\Phi_R'(h)|\nabla h|^2\,\dd\mu
 \le C_v\int_X(e^h+|B_F|^2)\,\dd\mu
       +R\|\nabla\chi_j\|_2\|\nabla h\|_2.
\end{equation}
Letting first $j\to\infty$ and then $R\to\infty$,
Lemma~\ref{lem:target-seeds} and monotone convergence yield
\begin{equation}\label{eq:z-energy}
 \int_X e^h|\nabla h|^2\,\dd\mu<\infty,
 \qquad z\in W^{1,2}(X).
\end{equation}
On $X_{\reg}$, the Cauchy-Schwarz inequality and the Bochner formula
give
\[
 \begin{gathered}
 \left|\sum_a(\Hess F_a)(\nabla F_a)\right|^2
 \le s|\Hess F|^2,\\
 \Delta z\ge\frac{|\Hess F|^2}{z^3}+\frac{s}{z}-A_v
 \ge-A_vz.
 \end{gathered}
\]
The capacity cutoffs \eqref{eq:base-capacity} extend this inequality
weakly to $X$:
\[
 \int_X\langle\nabla z,\nabla\varphi\rangle\,\dd\mu
 \le A_v\int_X z\varphi\,\dd\mu
 \qquad(0\le\varphi\in W^{1,2}(X)\cap L^\infty(X)).
\]
We now apply the global Sobolev inequality \eqref{eq:base-sobolev}
on $X$, supplied by \cite[Theorem 2.1, equation (2.4)]{GSS}.
For $r\ge2$ with $z\in L^r(X)$, set $z_L=\min(z,L)$ and test
with $z_L^{r-1}$. This gives
\[
 \frac{4(r-1)}{r^2}\int_X|\nabla(z_L^{r/2})|^2\,\dd\mu
 \le A_v\int_X z z_L^{r-1}\,\dd\mu
 \le A_v\int_X z^r\,\dd\mu.
\]
Applying \eqref{eq:base-sobolev} to $z_L^{r/2}$ and letting
$L\to\infty$, we obtain
\[
 \|z\|_{rq_S}
 \le\left[C_S\left(1+\frac{A_vr^2}{4(r-1)}\right)\right]^{1/r}
       \|z\|_r,
\]
see \cite[Lemma 3.3, equations (3.14)-(3.15)]{GSS}.
For $r_k=2q_S^k$, the series
$\sum_k(1+\log r_k)/r_k$ converges. Iterating from $r_0=2$, and
using $|\nabla f(y)|^2\le s(p(y))$, we therefore obtain
\[
 \|\nabla f\|_\infty\le\|\sqrt{s}\|_\infty
 \le\|z\|_\infty\le C_v\|z\|_2<\infty.
\]
Thus the Sobolev inequality is used only on $X$, while the sum over
the sheets controls the gradient at every point of $M$. Integration
along curves and completion then prove the assertion.
\end{proof}

\subsection{The Bochner inequality}

We first give a version of the cutoff argument in
\cite[Lemma 4.3 and Theorem 5.1]{HS} on bounded regions. The
proposition below requires only local gradient bounds, and proves the
Bochner inequality for compactly supported test functions.

\begin{proposition} 
\label{prop:local-bochner}
Let $(W,d_W,\mu_W)$ be a proper length metric-measure space with
full support, whose measure is finite on bounded sets. Suppose that
$W^\circ=W\setminus S$ is an open smooth Riemannian manifold of
real dimension $m\ge2$, that the distance agrees locally with its
Riemannian distance, and that $\mu_W$ is its Riemannian volume
extended by zero across $S$. Assume moreover that $S$ has locally
zero $W^{1,2}$-capacity, that the local Sobolev spaces agree
with the finite-energy spaces on $W^\circ$, with the Dirichlet energy, and that $\Ric\ge0$ on $W^\circ$.

Let $U\subset W$ be open, and suppose that
\begin{equation}\label{eq:core-regularity}
 \begin{gathered}
 f\in W^{1,2}_{\loc}(U)\cap C^\infty(U\cap W^\circ),\\
 \nabla f\in L^\infty_{\loc}(U),\qquad
 \Delta f\in W^{1,2}_{\loc}(U).
 \end{gathered}
\end{equation}
Then $|\nabla f|^2\in W^{1,2}_{\loc}(U)$, and
\begin{equation}\label{eq:BE-core}
 \frac12\int_U|\nabla f|^2\Delta\varphi\,\dd\mu_W
 \ge\int_U\varphi\left(\frac{(\Delta f)^2}{m}
       +\langle\nabla f,\nabla\Delta f\rangle\right)\,\dd\mu_W
\end{equation}
for every nonnegative $\varphi\in W^{1,2}(U)\cap L^\infty(U)$
with compact support in $U$ and $\Delta\varphi\in L^\infty(U)$.
Here both Laplacians are understood weakly on $U$.
\end{proposition}

\begin{proof}
Fix $V\Subset U$ containing $\operatorname{supp}\varphi$, and
choose $\eta\in\Lip_c(U)$, $0\le\eta\le1$, equal to one near
$\overline V$. Set
$K=\operatorname{supp}\eta$, and choose capacity cutoffs $\chi_j$,
vanishing near $S\cap K$, with
\[
 0\le\chi_j\le1,\qquad \chi_j\longrightarrow1\text{ a.e. on }K,
 \qquad \|\nabla\chi_j\|_{L^2(K)}\longrightarrow0.
\]
Testing the smooth Bochner inequality with $(\eta\chi_j)^2$
gives
\begin{equation}\label{eq:hs-hessian-cutoff}
 \begin{aligned}
 \frac12\int(\eta\chi_j)^2|\Hess f|^2\,\dd\mu_W
 &\le2\int|\nabla f|^2|\nabla(\eta\chi_j)|^2\,\dd\mu_W\\
 &\quad+\int(\eta\chi_j)^2|\nabla f|\,|\nabla\Delta f|\,\dd\mu_W.
 \end{aligned}
\end{equation}
The error near $S$ is bounded by
$C\|\nabla f\|_{L^\infty(K)}^2\|\nabla\chi_j\|_{L^2(K)}^2$.
Letting $j\to\infty$, we obtain
\[
 \int_V|\Hess f|^2\,\dd\mu_W
 \le C_\eta\int_K|\nabla f|^2\,\dd\mu_W
    +2\|\nabla f\|_{L^2(K)}\|\nabla\Delta f\|_{L^2(K)}<\infty.
\]
For $G=|\nabla f|^2$, it follows that
\[
 |\nabla G|\le2|\nabla f|\,|\Hess f|,
 \qquad G\in W^{1,2}_{\loc}(U).
\]
Here the Sobolev identification extends the finite-energy function
locally across $S$.

We now test the Bochner inequality with $\varphi\chi_j$, using
Sobolev approximation on $W^\circ$. The only singular cutoff error
satisfies
\[
 \left|\int_V\varphi\langle\nabla G,\nabla\chi_j\rangle
                      \,\dd\mu_W\right|
 \le\|\varphi\|_\infty\|\nabla G\|_{L^2(V)}
          \|\nabla\chi_j\|_{L^2(V)}\longrightarrow0.
\]
Since $\varphi$ has compact support in $V$, integration by parts
and the inequality $|\Hess f|^2\ge(\Delta f)^2/m$ give
\eqref{eq:BE-core}. We note that all the test functions are
supported in the fixed bounded set $K$.
\end{proof}

The proposition extends the Bochner inequality across $S$ for
functions with locally bounded gradient as in
\eqref{eq:core-regularity}. Here we do not  assert that a restricted
metric ball is an RCD space. The local-to-global theorem of
Honda-Sun \cite[Theorem 3.5]{HS} additionally assumes uniform local
volume doubling, a Poincar\'e inequality, and quantitative Lipschitz
estimates for harmonic functions. Since we have not established this
package independently on $\widehat M$, we carry out the remaining
global passage explicitly below.

\begin{proof}[Proof of Theorem~\ref{thm:lifting}]
The case $n=1$ is that of the flat cover of an elliptic curve. For
$n\ge2$, Section~\ref{sec:completion}, Lemma~\ref{lem:form},
Ricci-flatness and Proposition~\ref{prop:gradient} allow us to apply
Proposition~\ref{prop:local-bochner} with
$(W,W^\circ,m)=(\widehat M,M,2n)$ to every solution
$f=(1-\Delta)^{-1}v$ with $v\in C_c^\infty(M)$.
Indeed, $\Delta f=f-v\in W^{1,2}(\widehat M)$.

We now extend \eqref{eq:BE-core} to the full class of test functions,
using the global gradient bound of Proposition~\ref{prop:gradient}.
For such a solution $f$, take distance cutoffs $\eta_R$ equal to
one on $B_R(o)$ and supported in $B_{2R}(o)$, with
$0\le\eta_R\le1$ and $|\nabla\eta_R|\le2/R$.
Removing the singular cutoff in \eqref{eq:hs-hessian-cutoff} as
before gives
\[
 \frac12\int\eta_R^2|\Hess f|^2\,\dd\mu_M
 \le\frac C{R^2}\|\nabla f\|_2^2
       +\|\nabla f\|_2\|\nabla\Delta f\|_2.
\]
Letting $R\to\infty$, we obtain
\begin{equation}\label{eq:hs-hessian-global}
 \Hess f\in L^2(M),\qquad
 G=|\nabla f|^2\in W^{1,2}(\widehat M),
\end{equation}
since
\[
 \|G\|_2\le\|\nabla f\|_\infty\|\nabla f\|_2,
 \qquad
 \|\nabla G\|_2\le2\|\nabla f\|_\infty\|\Hess f\|_2.
\]
For any nonnegative $\varphi\in D(\Delta)\cap L^\infty$
with $\Delta\varphi\in L^\infty$, we repeat the Bochner test on the
smooth locus with $\varphi\eta_R\chi_j$. We first remove $\chi_j$
as above, the remaining cutoff error then satisfies
\[
 \left|\int\varphi\langle\nabla G,\nabla\eta_R\rangle
                         \,\dd\mu_M\right|
 \le\frac2R\|\varphi\|_2\|\nabla G\|_2\longrightarrow0.
\]
Thus \eqref{eq:BE-core} holds without the restriction to compact
support. For an arbitrary $f\in D(\Delta)$ with
$\Delta f\in W^{1,2}$, choose $v_j\in C_c^\infty(M)$ converging to
$v=(1-\Delta)f$ in $W^{1,2}$, and set $f_j=(1-\Delta)^{-1}v_j$.
By the spectral calculus,
\begin{equation}\label{eq:solution-core-convergence}
 \|f_j-f\|_{W^{1,2}}+\|\Delta f_j-\Delta f\|_{W^{1,2}}
 \le2\|v_j-v\|_{W^{1,2}}\longrightarrow0.
\end{equation}
Thus $|\nabla f_j|^2\to|\nabla f|^2$ and
$(\Delta f_j)^2\to(\Delta f)^2$ in $L^1$. For the mixed term,
\[
 \begin{aligned}
 &\|\langle\nabla f_j,\nabla\Delta f_j\rangle
       -\langle\nabla f,\nabla\Delta f\rangle\|_1\\
 &\quad\le\|\nabla(f_j-f)\|_2\|\nabla\Delta f_j\|_2
       +\|\nabla f\|_2\|\nabla\Delta(f_j-f)\|_2\longrightarrow0.
 \end{aligned}
\]
Since $\varphi$ and $\Delta\varphi$ are bounded, we may pass to the
limit in \eqref{eq:BE-core}, proving $\BE(0,2n)$.

It remains to verify the volume growth condition in
\cite[Definition 2.12]{HS}. Choose $o\in\widehat M$ and $R_0\ge1$
with $\Gamma\overline B_{R_0}(o)=\widehat M$. By properness,
\[
 \mathcal A=\{\gamma:d_M(o,\gamma o)\le2R_0+1\},\qquad
 |\mathcal A|<\infty.
\]
Subdividing geodesics into unit intervals and choosing nearby orbit
points, we find $|\gamma|_{\mathcal A}\le C(d_M(o,\gamma o)+1)$.
Note that $\mathcal A$ contains the finite stabilizer of $o$.
For $x\in B_R(o)$, choose $\gamma\in\Gamma$ with
$x\in\gamma\overline B_{R_0}(o)$. Since $\gamma$ is an isometry,
$\gamma\overline B_{R_0}(o)=\overline B_{R_0}(\gamma o)$, and
the triangle inequality gives
\[
 d_M(o,\gamma o)\le d_M(o,x)+d_M(x,\gamma o)\le R+R_0.
\]
Thus
\[
 B_R(o)\subset
 \bigcup_{d_M(o,\gamma o)\le R+R_0}
       \gamma\overline B_{R_0}(o).
\]
Since each translate has the same measure, counting these translates
by means of the word-length bound above gives
\begin{equation}\label{eq:volume-growth}
 \mu_M(B_R(o))
 \le\mu_M(\overline B_{R_0}(o))
      \sum_{0\le k\le C(R+1)}|\mathcal A|^k
 \le Ce^{CR}\qquad(R\ge1).
\end{equation}
Together with the Sobolev-to-Lipschitz property of
Lemma~\ref{lem:form}, this yields $\RCD(0,2n)$ by
\cite[Definitions 2.11-2.12]{HS}.
\end{proof}

\section{The fundamental group of the regular locus}\label{sec:decomposition}

\subsection{The Cheeger-Gromoll argument}

Let $p:M\to X_{\reg}$ be the universal cover, set
$\Gamma=\pi_1(X_{\reg})$, and denote by $\widehat M$ its metric
completion for the lifted Ricci-flat metric. By
Proposition~\ref{prop:completion} and Theorem~\ref{thm:lifting},
$\widehat M$ is a proper geodesic $\RCD(0,2n)$ space and a normal klt
complex space. The $\Gamma$-action is proper and cocompact by
holomorphic isometries, with $\widehat M/\Gamma=X$. Moreover,
$M\subset \widehat M_{\reg}$ and $\widehat M\setminus M$ is analytic of complex
codimension at least two.

\begin{lemma}\label{prop:cg-structure}
There is an isometric splitting
\begin{equation}\label{eq:cg-metric}
 \widehat M=\R^r\times C,\qquad 0\le r\le2n,
\end{equation}
with $C$ compact. The group $\Gamma$ contains a finite-index
subgroup $\Gamma_0\simeq\Z^r$ whose Euclidean projection is
an isomorphism onto a full translation lattice. In particular,
$\Gamma_0$ acts freely on $\widehat M$.
\end{lemma}
\begin{proof}
Following Cheeger--Gromoll \cite{CG}, repeated use of Gigli's
splitting theorem \cite{Gigli} gives $\widehat M=\R^r\times C$ with
$r\le2n$ maximal. Thus $C$ contains no line, and  is compact.

Let $\rho:\Gamma\to\operatorname{Isom}(\R^r)$ be the Euclidean
projection, $F=\ker\rho$, and $\Delta=\rho(\Gamma)$. Compactness
of $C$ and properness of the action imply that $\rho^{-1}(K)$
is finite for every compact $K\subset\operatorname{Isom}(\R^r)$.
Thus $F$ is finite and $\Delta$ is discrete; cocompactness
descends to $\Delta$. Bieberbach's theorem \cite{Bieberbach}
gives a finite-index translation lattice $L\subset\Delta$.

To remove the finite kernel, define
\[
 H=\{h\in\rho^{-1}(L):hf=fh\text{ for every }f\in F\},
 \qquad L_0=\rho(H).
\]
The subgroup $H$ is the kernel of the conjugation homomorphism
$\rho^{-1}(L)\to\operatorname{Aut}(F)$. Since $F$ is finite,
$H$ has finite index in $\Gamma$, and $L_0$ is a full lattice.
Moreover, $[H,H]\subset F\cap H\subset Z(H)$. Lift a basis
of $L_0$ to $h_1,\ldots,h_r\in H$ and put $e=|F|$. Their
commutators are central of orders dividing $e$, so the powers
$h_i^e$ commute. Consequently,
\[
 \Gamma_0=\langle h_1^e,\ldots,h_r^e\rangle
 \xrightarrow[\;\rho\;]{\;\simeq\;}eL_0\simeq\Z^r.
\]
Since $F$ is finite and $[\Delta:eL_0]<\infty$, $\Gamma_0$ has finite
index in $\Gamma$. Its nonidentity elements translate the
Euclidean factor nontrivially, so its action is free.
\end{proof}
\begin{remark} The above lemma should also follow the same argument as in Theorem~\ref{thm:etale} by first establishing a holomorphically isometric splitting of $\hat M$.
\end{remark}
\subsection{Proof of Theorem~\ref{thm:decomposition}}

\begin{proof}
Since $\widehat M_{\reg}\setminus M$ has complex codimension at least two,
transversality of loops and homotopies gives
$\pi_1(\widehat M_{\reg})=\pi_1(M)=\{1\}$. For a normal complex space,
the inclusion of the regular locus induces a surjection on
fundamental groups: the regular locus of its connected normal
universal cover is connected. Consequently,
\begin{equation}\label{eq:completion-pi1}
 \pi_1(\widehat M_{\reg})=\{1\},\qquad \pi_1(\widehat M)=\{1\}.
\end{equation}

Take $\Gamma_0$ from Lemma~\ref{prop:cg-structure} and set
$X'=\widehat M/\Gamma_0$. The free proper action makes $\widehat M\to X'$ a
covering by local biholomorphisms, and finite index makes $X'$
compact. The induced map $\tau:X'\to X$ is finite and \'etale over
$X_{\reg}$, hence quasi-\'etale. Finite Riemann existence
\cite[Expos\'e XII, Th\'eor\`eme 5.1]{SGA1} algebraizes its restriction over
$X_{\reg}$ and normalization over $X$ recovers $X'$ and makes it
projective. Quasi-\'etale pullback preserves klt singularities and
gives $K_{X'}=\tau^*K_X\equiv0$. Thus $X'$ is a projective
Calabi-Yau variety, and \eqref{eq:completion-pi1} identifies
$\widehat M\to X'$ as its universal cover, with $\pi_1(X')=\Gamma_0$.

Pullback of bounded potentials and uniqueness in \cite{EGZ}
identify $\omega':=\tau^*\omega$ with the Ricci-flat current in
$c_1(\tau^*L)$. Let $g'$ be its smooth Riemannian metric on
$X'_{\reg}$, and put
\[
 U=\tau^{-1}(X_{\reg})\subset X'_{\reg},\qquad
 S=X'_{\reg}\setminus U.
\]
Denote by $d'$ the Riemannian length distance on
$(X'_{\reg},g')$ and by $d_U$ that on $(U,g'|_U)$. Namely,
for $a,b\in U$,
\[
 \begin{aligned}
 d'(a,b)&=\inf_{\substack{\sigma:[0,1]\to X'_{\reg}\\
                         \sigma(0)=a,\ \sigma(1)=b}}L_{g'}(\sigma),\\
 d_U(a,b)&=\inf_{\substack{\sigma:[0,1]\to U\\
                          \sigma(0)=a,\ \sigma(1)=b}}L_{g'}(\sigma),
 \end{aligned}
\]
where the infima are over piecewise smooth curves. The set $S$ is
closed analytic in $X'_{\reg}$ with complex codimension at least two.
After smoothing corners, a curve with endpoints in $U$ can be
perturbed, fixing its endpoints, to be transverse to a Whitney
stratification of $S$. Since every stratum has real codimension at
least four, the perturbed curve avoids $S$. The perturbation may be
arbitrarily small in $C^1$, so its $g'$-length changes by arbitrarily
little. Together with $d'\le d_U$, this proves
\[
 d_U(a,b)=d'(a,b)\qquad(a,b\in U).
\]
Thus the inclusion of $U$ into $(X'_{\reg},d')$ is isometric and
has dense image, so the two spaces have isometric completions.

On the other hand, $U=M/\Gamma_0$, and lifting curves to $M$ gives
\[
 d_U([z],[w])=\inf_{\gamma\in\Gamma_0}d_M(z,\gamma w)
 \qquad(z,w\in M).
\]
Since $\widehat M$ is the completion of $(M,d_M)$, the completion of
$(U,d_U)$ is $\widehat M/\Gamma_0$ with the distance given by the same
formula for $z,w\in \widehat M$. Consequently, the completed Ricci-flat
length metric on $X'$ agrees with this quotient distance.

Theorem~\ref{thm:etale} and its proof now give a finite \'etale
cover $A\times Y\to X'$ and a holomorphic isometric splitting
\begin{equation}\label{eq:regular-cover-product}
 \widehat M\simeq\C^q\times Y,\qquad q=\dim A,
\end{equation}
where $A$ is an abelian variety and $Y$ is a simply connected
projective Calabi-Yau variety. The composition $A\times Y\to X$
is quasi-\'etale. Since $\widehat M_{\reg}=\C^q\times Y_{\reg}$,
\eqref{eq:completion-pi1} gives $\pi_1(Y_{\reg})=\{1\}$.
Also $\pi_1(A\times Y)=\Z^{2q}$ has finite index in
$\pi_1(X')=\Gamma_0\simeq\Z^r$, so $r=2q$.

We claim that every connected finite quasi-\'etale cover of $Y$ is
trivial. By Zariski purity \cite[Expos\'e X, Th\'eor\`eme 3.1]{SGA1},
its restriction over $Y_{\reg}$ is \'etale, and
this restriction is connected by normality. Simple connectedness of $Y_{\reg}$
forces its degree to be one, and the original finite map is then an
isomorphism. The claim is proved. Since $K_Y\equiv0$,
\cite[Theorem 1.2]{Gongyo} gives a trivial reflexive power of $K_Y$.
Applying the claim to the global index-one cover associated with
the smallest such power, we conclude that
$K_Y\sim0$. The Cartier property and the klt condition then imply
that $Y$ is canonical. Moreover, $q(Y)=0$ as in Section~\ref{sec:etale}, and since every such cover is trivial, we
obtain $\widetilde q(Y)=0$ by definition.

It remains to verify the torus dimension. The cover constructed above
satisfies $q(A\times Y)=q$, which gives $\widetilde q(X)\ge q$.
Conversely, let $X_1\to X$ be any connected finite quasi-\'etale
cover, and let $W$ be a connected component of the normalization of
$X_1\times_X(A\times Y)$. Both projections are finite, surjective and
quasi-\'etale. The induced cover of
$(A\times Y)_{\reg}=A\times Y_{\reg}$ is determined by a
finite-index sublattice of $\pi_1(A)$, and is therefore of the form
$A'\times Y_{\reg}$ for an isogeny $A'\to A$. By uniqueness of finite
normal extensions, $W=A'\times Y$, so that $q(W)=q$. The trace map
makes $H^1(X_1,\cO_{X_1})\to H^1(W,\cO_W)$ injective. Consequently,
$q(X_1)\le q$, which proves $q=\widetilde q(X)\le n$.

Finally, the inverse image $V$ of $X_{\reg}$ in $A\times Y$ is a
connected finite \'etale cover of $X_{\reg}$, and its complement in
$A\times Y_{\reg}$ has complex codimension at least two. By the
transversality observation above,
\[
 \pi_1(V)=\pi_1(A\times Y_{\reg})
         =\Z^{2q}=\Z^{2\widetilde q(X)}.
\]
This subgroup has finite index in $\pi_1(X_{\reg})$, which completes
the proof.
\end{proof}

\end{document}